\documentclass[10pt]{amsart}
\usepackage[utf8]{inputenc}
\usepackage{latexsym,amsmath,amssymb,amsmath,mathtools,tikz,microtype,comment,amsthm}
\usepackage[english]{babel}
\usepackage[margin=2.5cm]{geometry}
\usepackage{enumitem}
\usepackage{tikz-cd}
\usepackage{mathrsfs}
\usepackage{cancel}

\usetikzlibrary{calc,intersections,through,backgrounds,patterns}
\usetikzlibrary{decorations.markings}
\usetikzlibrary{decorations.pathreplacing}

\usepackage{pifont}
\usepackage{yfonts}

\usepackage[colorlinks=true, citecolor=blue, urlcolor=magenta, breaklinks=true]{hyperref}

\numberwithin{equation}{section}
\newtheorem{theorem}{Theorem}[section]
\newtheorem{lemma}[theorem]{Lemma}

\newtheorem{corollary}[theorem]{Corollary}

\theoremstyle{definition}
\newtheorem{definition}[theorem]{Definition} 
\newtheorem{remark}[theorem]{Remark}
\newtheorem{example}[theorem]{Example}

\DeclareMathOperator{\Min}{Min}
\newcommand{\Sym}[1]{Sym(#1)}

\newcommand{\K}{{\mathbb{K}}}
\newcommand{\N}{{\mathbb{N}}}

\definecolor{mygreen}{rgb}{0.0, 0.5, 0.0}
\usepackage{etoolbox}
\patchcmd{\subsection}{-.5em}{.5em}{}{}
\patchcmd{\subsection}{2}{3}{}{}

\title{Symbolic Powers of Toric Ideals}
\author{Giuseppe Favacchio and Graham Keiper}
\address[G. Favacchio]
	{Universit\`a degli studi di Palermo, Dipartimento di Ingegneria, Viale delle Scienze, Palermo, Italy}
\email{giuseppe.favacchio@unipa.it}

\address[G. Keiper]
	{Universit\`a degli studi di Catania, Viale Andrea Doria, Catania, Italy}
\email{graham.keiper@unict.it}


\date{May 11, 2025}
\begin{document}
\begin{abstract}
    This paper investigates the symbolic powers of toric ideals.  We first describe them in terms of the kernel of certain linear maps derived from the lattice structure of the toric ideal. 
Furthermore, we apply our results to show that symbolic powers of a toric ideal can also be expressed as saturations of regular powers  with the monomial given by the product of all the variables.
Finally, we conclude with a computationally significant result for computing symbolic powers of toric ideals. 
\end{abstract}
\maketitle
 
\section{Introduction}
Let $R$ be a standard graded polynomial ring over an algebraically closed field $\K$ of characteristic zero, we denote by $\mathfrak M$ the  irrelevant ideal of $R$ and  by $\boldsymbol{m}$ the product of all the variables in $R$. 
For a given ideal $I\subseteq R$ the \emph{$n^{\text{th}}$ regular power} is defined as \[I^{n}=\underbrace{I\cdot I\cdots I}_{n}.\] 
The \emph{$n^{\text{th}}$ symbolic power} of an ideal $I\subseteq R$ with no embedded primes is defined as
$$I^{(n)}=\bigcap_{\mathfrak q \in \Min(R/I)} (I^nR_{\mathfrak q})\cap R$$ where $\Min(R/I)$ are the minimal associated primes of $R/I$.

There are situations where the regular power does not behave as one would like it to, for example one can see that any $n^{\text{th}}$ order partial derivative of an element of $I^{n}$ is an element of $I$, however $I^{n}$ does not necessarily contain all these elements. In such situations we should instead investigate  $I^{(n)}$. Indeed,  the Nagata-Zariski Theorem claims, see for instance \cite[Theorem 3.14]{E95}, 

\noindent {\bf Nagata-Zariski Theorem.}
   \emph{ Let $I$ be a prime ideal of a polynomial ring $R$ over  an algebraically closed field~$\K$. Then 
    \[I^{(n)}=\left\{f\in R \ |\  \frac{\partial^{k}f}{\partial M}\in I, \text{ for all the monomials } M\in R_k\ \text{and}\ k < n  \right\}.\]}

We will use this result several times in the paper. It also captures the idea that the $n^{\text{th}}$ symbolic power of an ideal $I_V$ contains geometrically relevant information on the associated variety $V$, and in particular on its irreducible components.
Hence, for a radical ideal $I=I_V$ the Nagata–Zariski theorem  ensures that the $n^{\text{th}}$ symbolic power of $I$ is equal to 
$$ I^{(n)}=\bigcap_{p\in V} I_p^n $$
where $I_p$ is the ideal defining the point $p.$
 The study of symbolic powers of ideals has gained lot of attention because of deep connections with algebraic, geometric and combinatorial properties of $I$, see for instance \cite{bocci2016waldschmidt, BH10,carlini2020ideals, cooper2024splittings,  ficarra2025symbolic, geramita2017matroid, MMN2024formula} just to cite a few of them. See also \cite{Grifo24} for an introduction and an overview on several topics involving symbolic powers.

It follows from the Nagata-Zariski Theorem that the symbolic powers of $I$
are $\mathfrak M$-saturated ideals, i.e. $I^{(n)}:\mathfrak M=I^{(n)}$. Therefore, since it is always true that $I^{(n)}\supseteq I^n$,   we have
$$I^{(n)}\supseteq I^n: \mathfrak M^{\infty}.$$

The opposite containment does not always hold. It does, for instance,   when $I=I_Z$ is the ideal defining a zero-dimensional varieties $Z$, in this case we have $I_Z^{(n)}=I_Z^n:\mathfrak M$. 

When $I$ is a prime ideal we have that $I^n: \mathfrak M^{\infty}$ is a $I$-primary ideal and the following are equivalent (cf. \cite[Lemma 2.2]{Grifo24})
    \begin{itemize}
\item  The $n^{\text{th}}$ symbolic power of $I$ is the ideal $I^{(n)}=I^nR_I\cap R$;
\item    
    $I^{(n)}=\{f\in R \ |\ fg\in I^n\ \text{for some}\ g\notin I  \};$
        \item $I^{(n)}$ is the unique $I$-primary component in an irredundant
primary decomposition of $I^n$;
\item 
$I^{(n)}$ is the smallest $I$-primary ideal containing $I^n$.
    \end{itemize}
 Moreover, the equality $I^{(n)} = I^n$
is  equivalent to the condition that $I^n$
is a $I$-primary ideal.

In general, the $n^{\text{th}}$ symbolic power of $I$ can be computed starting from the $n^{\text{th}}$ regular power of $I$ and saturating it with respect to a certain radical ideal $J$, which usually is difficult to compute, see for instance \cite[Proposition 3.13]{E95} and \cite[Section 3]{HHT07}. For example, it is shown in \cite{MMN2024formula} that, when $I$ has height $c$ and under certain other assumptions, $J$ is the $c$-th fitting ideal of $I$.

In this paper we study symbolic powers of toric ideals. Several properties of toric ideals have been studied in recent years from different points of view, see
for instance \cite{bhaskara2023comparing, constantinescu2018, Adam23, F24, FHKV, FKV, garcia2023robustness,  K-thesis} and \cite[Section 5.3]{herzog2018binomial} for a more general overview on the topic.

In Section \ref{s.via the kernel } we show (Theorem \ref{t.t-th symb pow}) that the symbolic power of a toric ideal can be computed as a kernel of the linear maps  $\pi^{(t)}:  \K[e_1,\ldots, e_n]\to  \K^{n}\otimes\cdots\otimes\K^{n}$,   induced by    $\pi^{(t)}(e^{\alpha})= \underbrace{\alpha\otimes \cdots \otimes\alpha}_t,$ where $\alpha \in \mathbb N^n$ and $e^{\alpha}$ is the monomial $e_1^{\alpha_1}\cdots e_n^{\alpha_n}.$ 
 In Section \ref{s. via Saturations} we prove (Theorem \ref{t. sat second power} and Theorem \ref{t. sat t-th power}) that symbolic powers of a toric ideal $I_A$ in  $R=\K[e_{1}, \dots, e_{n}]$ are obtained by saturating the regular powers with $\boldsymbol{m}=e_{1}\cdots e_{n},$ that is 
$$I_A^{(t)}=I_A^{t}:\boldsymbol{m}^{\infty}.$$
Finally we describe (Remark \ref{r.improving})  how  an immediate result on saturation of homogeneous ideal (Lemma~\ref{l. main Saturations}) can be used  to improve the computation of symbolic powers of toric ideals.

\section{Symbolic powers of toric ideals via kernel of linear maps}\label{s.via the kernel }
In this section we characterize the symbolic powers of a toric ideal in terms of certain linear maps. We start by fixing  terminology and notation for toric ideals.

\begin{definition}[Toric Ideal] \label{d.toric}   
  Let $A\in \mathcal{M}_{m\times n}(\mathbb{N})$ be a matrix of non negative integers.  We denote its columns by $\varsigma_1, \ldots, \varsigma_n\in \mathbb N^{m}$.   We assume that the columns of $A$ are non-zero vectors. 
  Let $\varphi_A:\K[e_{1}, \dots, e_{n}]\rightarrow\K[x_{1}, \dots, x_{m}]$ be a graded ring homomorphism induced by  $\varphi(e_{i})= x^{\varsigma_{i}}$.
We define $I_{A}:=\ker(\varphi_{A})$ be the \emph{toric ideal defined by $A$}.
\end{definition}

\begin{definition}\label{d.fiber}
Given $A\in \mathcal{M}_{m\times n}(\mathbb{N})$ and $\sigma \in \N^m,$  the set of monomials in $\varphi_A^{-1}(x^{\sigma})\subseteq \K[e_{1}, \dots, e_{n}]$ is called the \emph{fiber} of $\sigma$. We denote by $V_{A,\sigma}$ the $\K$-vector space generated by the fiber of $\sigma$.

We say that $f\in\K[e_{1}, \dots, e_{n}]$ is  a \emph{ $\varphi_A$-homogeneous form} (or just \emph{ $\varphi_A$-homogeneous}) if there exists $\sigma \in \N^m$ such that $f\in V_{A,\sigma}$.   
\end{definition}

\begin{lemma}\label{l. NS cond for being in I}
   Given $A\in\mathcal{M}_{m\times n}(\mathbb{N})$, let $f=c_{1}e^{\beta_1}+\cdots +c_{r}e^{\beta_r}\in V_{A,\sigma}$ be a $\varphi_A$-homogeneous form. Then $f\in I_A$ if and only if  $c_1+\cdots+c_r=0$.
\end{lemma}
\begin{proof}
Note that $\varphi_A(f)=c_{1}x^{\sigma}+\cdots +c_{r}x^{\sigma}=x^{\sigma}(c_1+\cdots+c_r).$
Then
\[f\in I_A \Leftrightarrow \varphi_A(f)=0  \Leftrightarrow c_1+\cdots+c_r=0.\qedhere\]
\end{proof}

\begin{lemma}\label{l. generators in same fiber (t)}
     There exists a minimal set of generators for $I_A^{(t)}$ consisting of $\varphi_A$-homogeneous forms.
\end{lemma}
\begin{proof} 
The proof of the lemma is constructive describing an algorithm for obtaining such a set of $\varphi_{A}$-homogeneous generators.

 Let $\mathcal G=\{g_{1},\dots, g_{k}\}$ be a set of elements in $I^{(t)}$ constructed in the following way: 
    \begin{itemize}
        \item[$i)$]   Select $g_1\in I^{(t)}_{A}$ to be of minimal degree under the standard grading of $R=\K[e_{1}, \dots, e_{n}]$ which consists of a minimal number of terms;
        \item[$ii)$] for $i>1$, $g_i\in I_{A}^{(t)}$ is of minimal degree under the standard grading in $R$ 
        consisting of a minimal number of terms and such that $g_i\notin (g_1,\ldots, g_{i-1})$.
    \end{itemize}
    Since $R$ is Noetherian this process must terminate. By construction $\mathcal{G}$ is a minimal generating set of $I_A^{(t)}$. 
    
    We claim that the elements in $\mathcal G$ are $\varphi_A$-homogeneous. 

    Suppose for contradiction that $g_{i}=B_{1}+B_{2}+\cdots +B_{N}+C_{1}+\cdots +C_{M}\in\mathcal{G}$ where the $B_{j}$ and $C_{k}$ are the monomial terms of $g_{i}$ and where $B_{j}\in V_{A,\sigma}$ for $j\in[N]$ and $C_{k}\notin V_{A,\sigma}$ for $k\in[M]$. Then since $g_{i}\in I_{A}^{(t)}\subseteq I_{A}$ we have that

 \[\varphi_{A}(g_{i})=\sum_{j=1}^{N}\varphi_{A}(B_{j})+\sum_{k=1}^{M}\varphi_{A}(C_{k})=0.\]
 Since $C_{k}\notin V_{A,\sigma}$ we must have that $N$ is even and that $\sum_{j=1}^{N}\varphi_{A}(B_{j})=0$ and so $B_{1}+\cdots +B_{N}\in I_{A}$. This also implies that $C_{1}+\cdots +C_{M}\in I_{A}$

 By Nagata-Zariski Theorem, for any monomial $T=e^{\tau}$ of degree $\alpha$, where $\alpha<t$,
denoted by $\dfrac{\partial}{\partial T} = \dfrac{\partial^{\alpha}}{\partial e_1^{\tau_1}\cdots \partial e_1^{\tau_n}},$
we have
\[\dfrac{\partial g_{i}}{\partial T}=\dfrac{\partial B_{1}}{\partial T}+\cdots+\dfrac{\partial B_{N}}{\partial T}+\dfrac{\partial C_{1}}{\partial T}+\cdots +\dfrac{\partial C_{M}}{\partial T}\in I_A.\]
We note that $\dfrac{\partial}{\partial T}$ takes elements of $V_{A,\sigma}$ to elements of $V_{A,\sigma - A(\pi^{(1)}(T))}$. Hence $\dfrac{\partial B_{1}}{\partial T}, \dots, \dfrac{\partial B_{N}}{\partial T}\in V_{A,\sigma - A(\pi^{(1)}(T))}$ and by the same argument as before $\dfrac{\partial B_{1}}{\partial T}+\cdots+\dfrac{\partial B_{N}}{\partial T}\in I_{A}$ and $\dfrac{\partial C_{1}}{\partial T}+\cdots +\dfrac{\partial C_{M}}{\partial T}\in I_A$. Hence $B_{1}+\cdots +B_{N}\in I^{(t)}_{A}$ and $C_{1}+\cdots +C_{M}\in I^{(t)}_{A}$. But this contradicts our assumption that $g_{i}$ has a minimal number of terms since one of $B_{1}+\cdots +B_{N}$ or $C_{1}+\cdots +C_{M}$ has the same degree as $g_{i}$ but with fewer terms. 
\end{proof}

\begin{definition}\label{d.pi}
Consider $\K[e_{1}, \dots, e_{n}]$ as a vector space over $\K$. Let $s\ge1$ and \[\pi^{(s)}:  \K[e_1,\ldots, e_n]\to  \K^{n}\otimes\cdots\otimes\K^{n}=(\K^{n})^s\] be the $\K$-linear map induced by
    $$\pi^{(s)}(e^{\alpha})=\displaystyle{\bigotimes_s \alpha } =\underbrace{\alpha\otimes \cdots \otimes\alpha}_s.$$
\end{definition}
 We will denote the null tensor by $0$. We also define the map $\pi^{(0)}:\K[e_1,\ldots, e_n]\to \K$ to be the $\K$-linear map induced by $\pi^{(0)}(e^{\alpha})=1$, therefore  $\pi^{(0)}(c_1e^{\sigma_1}+ \cdots + c_re^{\sigma_r})=c_1+\cdots+ c_r$. 
\begin{remark}\label{r. NS condition for being in I}
    Note that, from Lemma \ref{l. NS cond for being in I},  a $\varphi_A$-homogeneous element $f$ belongs to $I_A$ if and only if $\pi^{(0)}(f)=0$.
\end{remark}

\begin{remark}
    We note that $\K[e_{1}, \dots, e_{n}]$ is infinite dimensional over $\K$. However, restricting the domain to $V_{A,\sigma}$ then $\pi^{(s)}|_{V_{A,\sigma}}$ is a map between finite dimensional vector spaces and so it has a matrix representation. 
\end{remark}

  \begin{lemma}\label{l. ker containement} 
For $A\in \mathcal{M}_{m\times n}(\mathbb{N})$ and $\pi^{(s)}$ as in Definition \ref{d.pi} we have
$$\ker(\pi^{(s)}|_{V_{A,\sigma}})\subseteq \ker(\pi^{(s-1)}|_{V_{A,\sigma}}).$$
\end{lemma}
\begin{proof} Let $f$ be a $\varphi_A$-homogeneous form, i.e.,
$f=c_1e^{\alpha_1}+\cdots+c_k e^{\alpha_k}\in V_{A,\sigma}$, for some $\sigma\in \N^{m}$, $\sigma\ne 0$.
Let $f\in \ker(\pi^{(s)}|_{V_{A,\sigma}})$, then 
\[c_1 \left(\bigotimes_{s} \alpha_1\right)+ \cdots+ c_k \left(\bigotimes_{s} \alpha_k\right)=0.\]
This implies that, for any $\underline{i}=(i_1,\ldots, i_{s-1})\in [n]^{s-1}$ 
$$c_1 \left(\prod_{j\in \underline{i}} \alpha_{1,j}\right) \alpha_1+ \cdots+ c_k \left(\prod_{j\in \underline{i}} \alpha_{k,j}\right) \alpha_k=0.$$
Applying $A$ we obtain
$$c_1 \left(\prod_{j\in \underline{i}} \alpha_{1,j}\right) A(\alpha_1)+ \cdots+ c_k \left(\prod_{j\in \underline{i}} \alpha_{k,j}\right)A(\alpha_k)=0.$$
Since $f\in V_{A,\sigma}$ we have $A(\alpha_i)=\sigma$ for $i\in[k]$. Therefore, we get 
\[\left(c_1 \left(\prod_{j\in \underline{i}} \alpha_{1,j}\right) + \cdots+ c_k \left(\prod_{j\in \underline{i}} \alpha_{k,j}\right)\right)\sigma=0,\ \ \ \ \forall \ \ \underline{i}=(i_1,\ldots, i_{s-1})\in [n]^{s-1}.\] Since $\sigma\ne 0$ we have 
\[c_1 \left(\prod_{j\in \underline{i}} \alpha_{1,j}\right) + \cdots+ c_k \left(\prod_{j\in \underline{i}} \alpha_{k,j}\right)=0,\ \ \ \ \forall \ \ \underline{i}=(i_1,\ldots, i_{s-1})\in [n]^{s-1},\]
that is
\[c_1 \left(\bigotimes_{s-1} \alpha_1\right)+ \cdots+ c_k \left(\bigotimes_{s-1} \alpha_k\right)=0.\qedhere\]
\end{proof}

\begin{theorem}\label{t.t-th symb pow} For $A\in \mathcal{M}_{m\times n}(\mathbb{N})$ let $I_{A}$ be the toric ideal of $A$ and $\pi^{(t)}$ be as in Definition \ref{d.pi}. Then 
$$I_A^{(t+1)}= \ker(\pi^{(t)})=\bigoplus_{\sigma \in \N^m} \ker(\pi^{(t)}|_{V_{A,\sigma}}).$$
\end{theorem}

\begin{proof}
The proof is by induction on $t$. We first prove the case $t=1$. By Lemma \ref{l. generators in same fiber (t)}, it is enough to restrict our attention to the elements $f\in I^{(2)}$ that are $\varphi_A$-homogeneous, i.e., $f=c_1e^{\alpha_1}+\cdots+c_k e^{\alpha_k}\in V_{A,\sigma}$, for some $\sigma\in \N^m$. By the Nagata-Zariski Theorem, $f\in I_A^{(2)}$ if and only if $\frac{\partial f}{\partial e_{j}} \in I_A$ for each $j=1,\ldots,n.$ We note that \begin{equation}\label{eq.partial}
    \frac{ \partial e^{\alpha_i}}{\partial e_{j}}=\begin{cases}
    0 & \text{if } \alpha_{i,j}=0\\
   \displaystyle\alpha_{i,j}\frac{e^{\alpha_i}}{e_{j}}= \alpha_{i,j}e^{\alpha_i- \pi^{(1)}(e_j)} & \text{if } \alpha_{i,j}>0\\
\end{cases}
\end{equation} so, for short we write
\[\frac{\partial f}{\partial e_{j}} =\sum^{k}_{i=1}
c_{i}\alpha_{i,j}e^{\alpha_{i}- \pi^{(1)}(e_j)}\in I_A.\] 
Note that $\frac{\partial f}{\partial e_{j}}$ is a $\varphi_A$-homogeneous form since it belongs to $V_{A,\sigma - \varsigma_j}$ (we had denoted by $\varsigma_j=A\pi^{(1)}(e_{j})$ the $j$-th column of $A$). Then, from Lemma \ref{l. NS cond for being in I}, 
$$\frac{\partial f}{\partial e_{j}} \in I_A \text{ is equivalent to}\ 
 \sum_{i=1}^kc_i\alpha_{i,j}=0,\ \text{for }\ j=1,\ldots,n, $$ that is equivalent to, $$\sum_{i=1}^{k}c_i\alpha_i=0.$$
The latter condition is precisely $f\in  \ker(\pi^{(1)}|_{V_{A,\sigma}}).$
This concludes the case $t=1$.

\vspace{.1in}

We now assume $$I_A^{(d+1)}=  \ker(\pi^{(d)})$$
for all $d< t$ and prove the equality for $d=t$. 

\vspace{.1in}

\noindent $\bullet$ First we prove that $I_A^{(t+1)}\subseteq  \ker(\pi^{(t)}).$ 

Again, by Lemma \ref{l. generators in same fiber (t)}, it is enough to consider a  $\varphi_A$-homogeneous form $f=c_{1}e^{\alpha_{1}}+\cdots +c_{k}e^{\alpha_{k}}\in I_A^{(t+1)}$, i.e. $f\in V_{A,\sigma}$ for some $\sigma \in \N^{m}$. By the Nagata-Zariski Theorem and the induction hypothesis, we have    $$f\in I_A^{(t+1)}\ \text{if and only if}\ \frac{\partial}{\partial e_{j}}f\in I_{A}^{(t)}=\ker(\pi^{(t-1)}),\ \ \text{for}\ j=1,\ldots, n;$$ 
with, as in \eqref{eq.partial},
\[\frac{\partial f}{\partial e_j}= 
c_{1}\frac{\partial e^{\alpha_{1}}}{\partial e_j}+\cdots +c_{k}\frac{\partial e^{\alpha_{k}}}{\partial e_j}= \sum_{i=1}^{k}
c_{i}\alpha_{i,j}e^{\alpha_{i}-\pi^{(1)}(e_{j})}\in V_{A,\sigma - \varsigma_{j}}.
\]
That is
\[
0=\pi^{(t-1)}\left(\frac{\partial f}{\partial e_j}\right)=\sum_{i=1}^{k}
\left(c_{i}\alpha_{i,j}\bigotimes_{t-1}\left(\alpha_{i}-\pi^{(1)}(e_{j})\right)\right).
\]
That is, 
$$\sum_{i=1}^kc_{i}\alpha_{i,j}\left(\alpha_{i}-\pi^{(1)}(e_{j})\right)_{u_{1}}\cdots \left(\alpha_{i}-\pi^{(1)}(e_{j})\right)_{u_{t-1}}\hspace{-15pt}=0\ \text{for any $u=(u_1,\ldots, u_{t-1})\in [n]^{t-1}$}.$$
Let $b_j$ be the number of entries in $u=(u_1,\ldots, u_{t-1})$ that are equal to $j$.  If $b_j=0$ then 
$$\sum_{i=1}^{k}c_{i}\alpha_{i,j}\left(\alpha_{i}-\pi^{(1)}(e_{j})\right)_{u_{1}}\cdots \left(\alpha_{i}-\pi^{(1)}(e_{j})\right)_{u_{t-1}}=\sum_{i=1}^{k}c_{i}\alpha_{i,j}\alpha_{i,u_1}\cdots \alpha_{i,u_{t-1}}=0
.$$
If $b_j>0$ then  we assume, since the other cases are similar, that $u=(j, \ldots, j, u_{b_{j+1}},\ldots, u_{t-1})$ and so we get 
\[\begin{array}{rl}
    0=&\displaystyle  \sum_{i=1}^{k}c_{i} \alpha_{i,j} (\alpha_{i,j}-1)^{b_j}\alpha_{i,u_{b_{j+1}}}\cdots \alpha_{i,u_{t-1}} \\
     =& \displaystyle \sum_{i=1}^{k}c_{i}\alpha_{i,j}\left(\sum_{h=0}^{b_j}  \binom{b_j}{h}(-1)^{b_j-h}(\alpha_{i,j})^{h}\right)\alpha_{i,u_{b_{j+1}}}\cdots \alpha_{i,u_{t-1}} \\
     =& \displaystyle \sum_{h=0}^{b_j} \binom{b_j}{h}(-1)^{b_j-h} \sum_{i=1}^{k}c_{i}\alpha_{i,j} (\alpha_{i,j})^{h}\alpha_{i,u_{b_{j+1}}}\cdots \alpha_{i,u_{t-1}}\\
=  &  \displaystyle\sum_{h=0}^{b_j-1} \binom{b_j}{h}(-1)^{b_j-h} \sum_{i=1}^{k}c_{i} (\alpha_{i,j})^{h+1}\alpha_{i,u_{b_{j+1}}}\cdots \alpha_{i,u_{t-1}}+  \sum_{i=1}^{k}c_{i} (\alpha_{i,j})^{b_j+1}\alpha_{i,u_{b_{j+1}}}\cdots \alpha_{i,u_{t-1}}.\\   
\end{array}
\]
However, for $h<b_j$, by induction $f\in \ker(\pi^{(t-1-b_j+h)})$, therefore
$$\sum_{i=1}^{k}c_{i} (\alpha_{i,j})^{h+1}\alpha_{i,u_{b_j+1}}\cdots \alpha_{i,u_{t-1}}=0 \ \ \text{for}\ \ h<b_j.$$
Hence, from the above computation, we get 
$$\sum_{i=1}^{k}c_{i} (\alpha_{i,j})^{b_j+1}\alpha_{i,u_{b_j+1}}\cdots \alpha_{i,u_{t-1}}=0.$$
This shows that $f\in \ker(\pi^{(t)}).$

\vspace{.1in}

\noindent $\bullet$ Now we prove that $\ker(\pi^{(t)} )\subseteq I_A^{(t+1)}.$

Let  $f=c_{1}e^{\alpha_{1}}+\cdots +c_{k}e^{\alpha_{k}}\in  \ker(\pi^{(t)})$.  From Lemma \ref{l. ker containement} we have $f\in \ker(\pi^{(d)})$ for any $d\le t$ that is,
\[\sum_{i=1}^{k} c_i \bigotimes_{d} \alpha_i=0\]
and, therefore,
\begin{equation}\label{eq.ker pi d}\sum_{i=1}^k c_i \alpha_{i,j_1}\cdots \alpha_{i,j_d}=0 \ \text{for any $(j_1,\ldots, j_d)$ and $d\le t$.}   \end{equation} 
 In order to show that $f\in I^{(t+1)}_A$ it is enough to prove that $\dfrac{\partial}{\partial e_j}f\in I^{(t)}_A$ for any $j\in [n].$ 
We have, as in~\eqref{eq.partial},
\[\frac{\partial f}{\partial e_j}= 
c_{1}\frac{\partial e^{\alpha_{1}}}{\partial e_j}+\cdots +c_{k}\frac{\partial e^{\alpha_{k}}}{\partial e_j}= \sum_{i=1}^{k}
\alpha_{i,j}c_{i}e^{\alpha_{i}-\pi^{(1)}(e_{j})}.
\]
By induction, $I^{(t)}_A=\ker(\pi^{(t-1)})$, so it is enough to show that $$\pi^{(t-1)}\left(\frac{\partial f}{\partial e_{j}}\right)=0.
$$
We have
$$\pi^{(t-1)}\left(\frac{\partial f}{\partial e_{j}}\right)=\sum_{i=1}^{k}c_{i}
\alpha_{i,j}\bigotimes_{t-1}\left(\alpha_{i}-\pi^{(1)}(e_{j})\right)$$
Then, it is enough to prove that 
$$\sum_{i=1}^{k}c_{i}\alpha_{i,j}\left(\alpha_{i}-\pi^{(1)}(e_{j})\right)_{u_1}\cdots \left(\alpha_{i}-\pi^{(1)}(e_{j})\right)_{u_{t-1}}\hspace{-15pt}=0\ \text{for any $u=(u_1,\ldots, u_{t-1})\in[n]^{t-1}$}.$$
So, fix $u=(u_1,\ldots, u_{t-1})$ and let $b_j$ be the number of occurrences of $j$ in such vector. 
If $b_j=0$ then 
$$\sum_{i=1}^{k}c_{i}\alpha_{i,j}\left(\alpha_{i}-\pi^{(1)}(e_{j})\right)_{u_1}\cdots \left(\alpha_{i}-\pi^{(1)}(e_{j})\right)_{u_{t-1}}=\sum_{i=1}^{k}c_{i}\alpha_{i,j}\alpha_{i,u_1}\cdots \alpha_{i,u_{t-1}}
$$ which is $0$ since $f\in \ker(\pi^{(t)})$.
If $b_j>0$, then  we assume (the other cases are analogous) that $u=(j, \ldots, j, u_{b_{j+1}},\ldots, u_{t-1})$
We have
\begin{align*}
     \sum_{i=1}^{k}\alpha_{i,j}c_{i}&\left(\alpha_{i}-\pi^{(1)}(e_{j})\right)_{u_1}\cdots\left(\alpha_{i}-\pi^{(1)}(e_{j})\right)_{u_{t-1}}\\=&\sum_{i=1 }^{k}c_{i}\alpha_{i,j}(\alpha_{i,j}-1)^{b_j}\alpha_{i,u_{b_j+1}}\cdots\alpha_{i,u_{t-1}} \\
     =& \sum_{h=0}^{b_j} \binom{b_j}{h}(-1)^{b_j-h}\sum_{i=1 }^kc_{i}\alpha_{i,j}^{h+1}\alpha_{i,u_{b_j+1}}\cdots\alpha_{i,u_{t-1}}.
\end{align*}
This concludes the proof since, by equation \eqref{eq.ker pi d}, we have  $\displaystyle\sum_{i=1 }^kc_{i}\alpha_{i,j}^{h+1}\alpha_{i,u_{b_j+1}}\cdots\alpha_{i,u_{t-1}}=0,$ for $h\le b_j.$
\end{proof}

\begin{example}
    Let $A$ be the incidence matrix of the 4-cycle graph, $$A=\left(
    \begin{array}{cccc}
         1&0&0&1 \\
         1&1&0&0\\
         0&1&1&0\\
         0&0&1&1\\
    \end{array}\right).$$
Then the toric ideal $I_A=\ker(\varphi_A)$ is generated by only one element, namely, $f=e_1e_3-e_2e_4$. Indeed  $\varphi_A(f)=x_1x_2x_3x_4-x_1x_2x_3x_4=0$.
The form $f\in V_{A,(1,1,1,1)}$ so it is $\varphi_A$-homogeneous, moreover we have $\pi^{(0)}(f)=1- 1= 0$ and  $\pi^{(1)}(f)=(1,0,1,0)- (0,1,0,1)\neq 0.$

Let $\sigma=(2,2,2,2)$, note that   
$f^2\in V_{A,\sigma}=\langle e_1^2e_3^2, e_1e_2e_3e_4,e_2^2e_4^2 \rangle$ and 
$$\pi^{(1)}(f^2)=\pi^{(1)}(e_1^2e_3^2-2e_1e_2e_3e_4+e_2^2e_4^2)=(2,0,2,0) -2(1,1,1,1) +(0,2,0,2) =0.$$
As computed in the proof of Theorem \ref{t.t-th symb pow}, we have
$$\frac{\partial}{\partial e_{1}}f^2=2e_1e_3^2-2e_2e_3e_4\in V_{A,\sigma-(1,1,0,0)}$$
where $(1,1,0,0)$ is the first column in $A$. 
\end{example}

\begin{example}[cf. Example 2.18. in \cite{Grifo24}]\label{e. k33 second power}
    
Let $I_A=I_{K_{3,3}}$ be the toric ideal of the complete bipartite graph $K_{3,3}$. 
 The incidence matrix of $K_{3,3}$ is

$$A=\begin{pmatrix}
1 & 1 & 1 & 0 & 0 & 0 & 0 & 0 & 0  \\
0 & 0 & 0 & 1 & 1 & 1 & 0 & 0 & 0  \\
0 & 0 & 0 & 0 & 0 & 0 & 1 & 1 & 1  \\
1 & 0 & 0 & 0 & 1 & 0 & 1 & 0 & 0  \\
0 & 1 & 0 & 1 & 0 & 1 & 0 & 1 & 0  \\
0 & 0 & 1 & 1 & 0 & 0 & 0 & 0 & 1 
\end{pmatrix}.
$$
The ideal $I_A$ is generated by the nine two by two minors of the matrix
$$M=\left( \begin{array}{ccc}
   e_{1}  & e_{2}  & e_{3}  \\
   e_{4}  & e_{5}  & e_{6}  \\
   e_{7}  & e_{8}  & e_{9}  \\
\end{array} \right).$$
Consider the $\varphi_A$-homogeneous form given by the determinant of $M$
$$ 
f=|M|=e_{1}e_{5}e_{9}-e_{1}e_{6}e_{8}+
e_{2}e_{6}e_{7}-e_{2}e_{4}e_{9}+
e_{3}e_{4}e_{8}-e_{3}e_{5}e_{7}.
$$
We note that since $\deg(f)=3$ so clearly  $f\notin I^2_A.$ 
To check that $f\in I^{(2)}_A$ we use Theorem~\ref{t.t-th symb pow}. Since $$\begin{array}{rcl}
    \pi^{(1)}(f)&=& (1,0,0,0,1,0,0,0,1)-(1,0,0,0,0,1,0,1,0)+  \\
     && (0,1,0,0,0,1,1,0,0)-(0,1,0,1,0,0,0,0,1)+\\
     &&(0,0,1,1,0,0,0,1,0)-(0,0,1,0,1,0,1,0,0)\\
     &=&(0,0,0,0,0,0,0,0,0).
\end{array}$$
Thus $f\in \ker(\pi^{(1)})=I_A^{(2)}.$
\end{example}

\section{Symbolic powers of toric ideals via saturations}\label{s. via Saturations}  
Let $R=\K[e_{1}, \dots, e_n]$,  and denote by $\boldsymbol{m}=e_{1} \cdots e_n\in R$ and $\mathfrak M=(e_{1}, \dots, e_n)\subset R$. 

Recall that every $u\in \mathbb{Z}^{n}$ can be written uniquely as $u=u^+-u^-$, where $u^+$ and $u^-$ are non-negative integer vectors and have disjoint support. We set $f_u=e^{u^+}-e^{u^-}.$

Let $A\in  \mathcal{M}_{m\times n}(\mathbb{N})$ be a matrix of non-negative integers and, as in Definition \ref{d.toric}, let $I_A$ denote the toric ideal of $A$.  We assume that the columns of $A$ are non-zero vectors.  

The matrix $A$ defines, via the usual multiplication, a linear map $A: \mathbb Z^n\to \mathbb Z^m$.  We denote by $\ker(A)\subseteq \mathbb Z^n$ the kernel of the linear map defined by $A.$  For a subset $\mathcal C$ of $\ker(A)$ we define the following ideal, which is generated by the binomial forms corresponding to the elements in $\mathcal C$,  
$$J_{\mathcal C}= (f_v\ |\ v\in \mathcal C)\subseteq R.$$
Note that $J_{\mathcal C} \subseteq I_{A}.$ We recall a useful result of Sturmfels from \cite{sturmfels1996}.

\begin{lemma}[cf. Lemma 12.2 in \cite{sturmfels1996}]\label{l. Sturm96}
    A subset $\mathcal{C}$ spans the lattice $\ker(A)$ if and only if 
    \[\left(J_{\mathcal{C}}:\boldsymbol{m}^{\infty}\right)=I_{A}.\]
\end{lemma}

\begin{remark}
In   the proof of Lemma 12.2 in \cite{sturmfels1996}, the author considers a set $\mathcal{C}=\{v_{1},\dots, v_{r}\}\subset \ker(A)$ which generates the lattice, and shows that for each $u\in\ker(A)$, integer vector, there are integers  $\lambda_{i}$ such that $$u=\sum^{r}_{i=1}\lambda_{i}v_{i}.$$ Moreover, this identity is equivalent to having 
\[
\frac{e^{u^{+}}}{e^{u^{-}}}-1=\prod^{r}_{i=1}\left(\frac{e^{v^{+}_{i}}}{e^{v^{-}_{i}}}\right)^{\lambda_{i}}-1.
\] 
Hence a monomial multiple of $f_{u}$ lies in $J_{\mathcal C},$ in particular
\[
\left(\prod^{r}_{i=1}(e^{\lambda_iv^{-}_{i}})\right) f_u=\prod^{r}_{i=1}\left(e^{\lambda_iv^{+}_{i}}\right)-\prod^{r}_{i=1}\left(e^{\lambda_iv^{-}_{i}}\right)=e^{\sum^{r}_{i=1}\lambda_iv^{+}_{i}}-e^{\sum^{r}_{i=1}\lambda_iv^{-}_{i}}.
\] 
Moreover,  the following formula holds  for any $u_1,u_2\in \mathbb Z^n$
\begin{equation}\label{eq.comb with monom}
      \begin{array}{rl}
           e^{u_{1}^{+}+u_{2}^{+}}-e^{u_{1}^{-}+u_{2}^{-}}=& e^{u_{2}^{+}}(e^{u_{1}^{+}}-e^{u_{1}^{-}})+e^{u_{1}^{-}}(e^{u_{2}^{+}}-e^{u_{2}^{-}}) \\
           =& e^{u_{2}^{+}}f_{u_1}+e^{u_{1}^{-}}f_{u_2}.
      \end{array}
\end{equation}
The binomial $e^{\sum^{r}_{i=1}v^{+}_{i}}-e^{\sum^{r}_{i=1}v^{-}_{i}}$ can be written as
\[
e^{v_1^++\sum^{r}_{i=2}v^{+}_{i}}-e^{v_1^+-\sum^{r}_{i=2}v^{-}_{i}}= e^{\sum^{r}_{i=2}v^{+}_{i}}(e^{v_1^+}-e^{v_1^-})+ e^{v_1^-}(e^{\sum^{r}_{i=2}v^{+}_{i}}-e^{\sum^{r}_{i=2}v^{-}_{i}})
\]
Therefore, by using formula \eqref{eq.comb with monom}, we can iterate the above equality to  obtain that  $f_{u}$ multiplied by some monomial is sum of the $f_{v_{i}}$ with monomial coefficients. In particular, we get

\begin{equation}\label{eq.comb with monom 2}
   \displaystyle e^{\sum^{r}_{i=1}v^{+}_{i}}-e^{\sum^{r}_{i=1}v^{-}_{i}}= \sum_{j=1}^r \left( e^{\sum^{j-1}_{i=1}v^{-}_{i}\ +\ \sum^{r}_{i=j+1}v^{+}_{i}}f_{v_j}   \right).
\end{equation}
\end{remark}

The following result is a consequence of the remark above.
\begin{corollary}\label{c.monomial multiple}
   
Let $u\in \langle v_1,\ldots, v_r \rangle\subseteq \ker(A)$,   then $$e^{\alpha}f_u=e^{\beta_1} f_{v_1}+\cdots+ e^{\beta_r} f_{v_r}$$ for some $\alpha, \beta_1, \ldots, \beta_r\in \mathbb N^n$. 
\end{corollary}

\begin{proof}
Let $u_1,\ldots, u_r$ span $\ker(A)$ and  $u=\sum \lambda_i u_i$ for some $\lambda_i\in \mathbb Z.$ Set $v_i=\lambda_i u_i$ and let 
   $\alpha$ be such that  $$\sum^{r}_{i=1}v^{+}_{i}= \left(\sum^{r}_{i=1}v_{i}\right)^{+}+\alpha \ \text{ and } \  \sum^{r}_{i=1}v^{-}_{i}= \left(\sum^{r}_{i=1}v_{i}\right)^{-}+\alpha.$$ Then, the statement follows from equation \eqref{eq.comb with monom 2}, which gives
   \[e^{\alpha}f_u=e^{\beta_1} f_{\lambda_1u_1}+\cdots+ e^{\beta_r} f_{\lambda_ru_r}\qedhere\]
\end{proof}

\begin{remark}\label{r. f lambda u} In the main results of this section will be useful to express a binomial form $f_{\lambda u}$, where $\lambda\in \mathbb Z$, in terms of $f_u$. We have
     $$\begin{array}{rl}
         f_{\lambda u}&= sgn(\lambda)(e^{\lambda u^+}-e^{\lambda u^-})  \\[10pt]
          & =sgn(\lambda)\left(\sum\limits_{i=0}^{|\lambda|-1} {(e^{iu^+})} {(e^{(|\lambda|-1-i)u^-})}   \right)f_u\\[15pt]
          &=sgn(\lambda)\left(\sum\limits_{i=0}^{|\lambda|-1} {e^{iu + (|\lambda|-1)u^-}}   \right)f_u\\[15pt]
          &=sgn(\lambda)\left(\sum\limits_{i=0}^{|\lambda|-1} {e^{iu}}   \right)e^{(|\lambda|-1)u^-}f_u.
     \end{array}$$
\end{remark}
We show now that the second symbolic power of a toric ideal $I_A$ can be obtained by saturating $I_A^2$ with $\boldsymbol{m}.$

\begin{theorem}\label{t. sat second power}
    Let $I_A$ be a toric ideal of $R=\K[e_{1}, \dots, e_{n}]$ and $\boldsymbol{m}=e_{1}\cdots e_{n}$. Then $$\left(I_A^{2}:\boldsymbol{m}^{\infty}\right)=I_A^{(2)}.$$
\end{theorem}
\begin{proof} 
Since $I_A$ is a prime ideal, the fact that $\left(I_A^{2}:\boldsymbol{m}^{\infty}\right)\subseteq I_A^{(2)}$ follows from 
   $$I_A^{(t)}=\{f \in R\ |\ fg \in I_A^t\ \text{for some }\ g \notin I_A \},$$ 
see for instance \cite[Lemma 2.2]{Grifo24}.

Now we show that  $\left(I_A^{2}:\boldsymbol{m}^{\infty}\right)\supseteq I_A^{(2)}$.

By Lemma \ref{l. generators in same fiber (t)}, there is a minimal set of generators for $I_A^{(2)}$ which are $\varphi_A$-homogeneous. Let $f$ be one of them, i.e., 
    $f=c_1e^{\beta_1}+\cdots +c_re^{\beta_r},$ for some $\beta_i\in \mathbb N^n$ and 
    $\varphi_A(e^{\beta_1})=\cdots=\varphi_A(e^{\beta_r}).$
(Note that $r>1$ since there are no monomials in $I_A$ ).
       Also, we can assume $c_i\in \mathbb Z.$
  
   By Lemma \ref{l. NS cond for being in I} we have $c_1+\cdots+c_r= 0$ and, without loss of generality, we assume $c_r\neq 0.$ Then
     $$f=c_1(e^{\beta_1}-e^{\beta_r})+\cdots +c_{r-1}(e^{\beta_{r-1}}-e^{\beta_r})$$
and $\beta_1-\beta_r, \ldots, \beta_{r-1}-\beta_r \in \ker(A).$ Let $\{u_1,\ldots, u_s\}$ be a minimal set of vectors which span $\ker(A)$.
Thus, by Lemma \ref{l. Sturm96}, there exist $v_{i,j}=\lambda_{i,j}u_j$ and a large enough non negative integer $a \in \mathbb N$ such that
$$\begin{array}{rl}
     \boldsymbol{m}^{a}(e^{\beta_1}-e^{\beta_r})= & e^{w_{1,1}}f_{v_{1,1}}+ \cdots+ e^{w_{1,s}}f_{v_{1,s}} \\
  \boldsymbol{m}^{a}(e^{\beta_2}-e^{\beta_r})=   & e^{w_{2,1}}f_{v_{2,1}}+ \cdots+ e^{w_{2,s}}f_{v_{2,s}}\\
  \vdots\\
  \boldsymbol{m}^{a}(e^{\beta_{r-1}}-e^{\beta_r})=   & e^{w_{r-1,1}}f_{v_{r-1,1}}+ \cdots+ e^{w_{r-1,s}}f_{v_{r-1,s}}\\
\end{array}.$$
Then we have 
\begin{equation}\label{eq. m^uf}
\boldsymbol{m}^af=c_1\left(\sum_{j=1}^{s} e^{w_{1,j}}f_{v_{1,j}}\right)+ \cdots + c_{r-1}\left(\sum_{j=1}^{s} e^{w_{r-1,j}}f_{v_{r-1,j}}\right).   
\end{equation}
Since $\boldsymbol{m}^af\in I^{(2)}$, by  Theorem \ref{t.t-th symb pow}, we have $$c_1\left(\sum_{j=1}^{s} ({w_{1,j}}+{v_{1,j}^+})- ({w_{1,j}}+{v_{1,j}^-})\right)+ \cdots + c_{r-1}\left(\sum_{j=1}^{s} ({w_{r-1,j}}+{v_{r-1,j}^+})- ({w_{r-1,j}}+{v_{r-1,j}^-})\right)= 0,$$ that is
$$c_1 \sum_{j=1}^s(v_{1,j})+ \cdots + c_{r-1} \sum_{j=1}^s(v_{r-1,j})=c_1 \sum_{j=1}^s(\lambda_{1,j}u_{j})+ \cdots + c_{r-1} \sum_{j=1}^s(\lambda_{r-1,j}u_{j}) = 0.$$
And since $u_1,\ldots, u_s$ are linearly independent we get
\begin{equation}\label{eq. sum lambda is 0}
    \sum_{i=1}^{r-1}c_i\lambda_{i,1}=\cdots=  \sum_{i=1}^{r-1}c_{i}\lambda_{i,s} = 0.
\end{equation} 
Now we write
$$\boldsymbol{m}^af=\sum_{t=1}^s (c_1e^{w_{1,t}}f_{v_{1,t}}+\cdots +c_{r-1}e^{w_{r-1,t}}f_{v_{r-1,t}}).$$
By Remark \ref{r. f lambda u}, we have 
$$\boldsymbol{m}^af=g_1f_{u_1}+\cdots g_sf_{u_s}$$
for certain forms $g_1,\ldots, g_s$. In particular, from the above computation, $g_1$ is the following

{\small$$g_1=\sum_{i=1}^{r-1}\left(c_ie^{w_{i,1}}sgn(\lambda_{i,1})\left(\sum_{j=0}^{|\lambda_{i,1}|-1} {e^{ju_1}}   \right)e^{(|\lambda_{i,1}|-1){u_i}^-}\right).$$
}
Moreover we have
\[\begin{array}{rl}
     \pi^{(0)}(g_1)=&\displaystyle \pi^{(0)}\left(\sum_{i=1}^{r-1}\left(c_ie^{w_{i,1}}sgn(\lambda_{i,1})\left(\sum_{j=0}^{|\lambda_{i,1}|-1} {e^{ju_1}}   \right)e^{(|\lambda_{i,1}|-1){u_i}^-}\right)\right)= \sum_{i=1}^{r-1}c_isgn(\lambda_{i,1})|\lambda_{i,1}| \\
     =&\displaystyle \sum_{i=1}^{r-1}c_i\lambda_{i,1}=0,
\end{array}\]
where the last equality is the condition in \eqref{eq. sum lambda is 0}. Hence, by Remark \ref{r. NS condition for being in I} we have $g_1\in I_A$. In a similar way we see that also $g_2, \ldots, g_s\in I_A$ and therefore $\boldsymbol{m}^af\in I^2_A.$
\end{proof}

\begin{example}
Here we illustrate the proof of Theorem \ref{t. sat second power} with respect the toric ideal $I_A$ as in Example~\ref{e. k33 second power}.  
A set of generators for $\ker(A)$ is
$$\begin{array}{rrrrrrrrrrr}
&{}_1&{}_2&{}_3&{}_4&{}_5&{}_6&{}_7&{}_8&{}_9\ \\
     u_1=&(\ 1,&  -1,&  0, & -1, & 1, & 0, & 0, & 0,&  0\ ) \\
     u_2=&(\ 0,& 1,& -1,& 0, &-1,& 1,& 0, &0, &0\ )\\
     u_3=&(\ 0, &0,& 0,& 1,& 0,& -1,& -1,& 0, &1\ )\\
u_4=&(\ 0, &0, &0,& 0,& 1, &-1,& 0, &-1,& 1\ )\\
\end{array}$$ 
The form $f$ is $\varphi_A$-homogeneous, indeed each monomial in $f$ is mapped by $\varphi_A$ into $x_1x_2x_3y_1y_2y_3$.

In order to write $f$ as in the proof of Theorem \ref{t. sat second power}, we set $$\begin{array}{rrrrrrrrrrr}
&{}_1&{}_2&{}_3&{}_4&{}_5&{}_6&{}_7&{}_8&{}_9\ \\

 v_1=&(\ 1,&0,&0,&0,&1,&0,&0,&0,&1\ )\\
v_{2}=&(\ 1,&0,&0,&0,&0,&1,&0,&1,&0\ )\\

  v_{3}=&(\ 0,&1,&0,&0,&0,&1,&1,&0,&0\ )\\
  v_4=&(\ 0,&1,&0,&1,&0,&0,&0,&0,&1\ )\\  

  v_5=&(\ 0,&0,&1,&1,&0,&0,&0,&1,&0\ )\\ 
 v_{6}=&(\ 0,&0,&1,&0,&1,&0,&1,&0,&0\ )  
\end{array}
$$ 
Thus, we have 
$$f=e^{v_1}-e^{v_2}+ e^{v_3}-e^{v_4}+e^{v_{5}}-e^{v_{6}}.$$
Since $v_1-v_2+v_3-v_4+v_5-v_6=0$, from Theorem \ref{t.t-th symb pow}, we have $f\in I_A^{(2)}$. 
We show now that $f\in I_A^2:\boldsymbol{m}^{\infty}$
$$f=(e^{v_1}-e^{v_{6}})-(e^{v_2}-e^{v_{6}})+ (e^{v_3}-e^{v_{6}})-(e^{v_4}-e^{v_{6}})+(e^{v_{5}}-e^{v_{6}}).$$
We have
$$\begin{array}{rl}
v_1-v_6=&u_1+u_2+u_3  \\
v_2-v_6=&u_1+u_2+u_3-u_4 \\ 
v_3-v_6=&u_2\\
v_4-v_6=&u_2+u_3\\
v_5-v_6=&u_3-u_4
\end{array}$$
and 

$$\begin{array}{rl}
e_6(e^{v_1}-e^{v_6})=&e_6e_9f_{u_1}+e_4e_9f_{u_2}+e_3e_5f_{u_3}  \\

e_5e_9(e^{v_2}-e^{v_6})=&
e_6e_8e_9f_{u_1}+
e_4e_8e_9f_{u_2}+
e_3e_5e_8f_{u_3}-e_3e_5e_7f_{u_4} \\ 

e^{v_3}-e^{v_6}=&e_7f_{u_2}\\

e_6(e^{v_4}-e^{v_6})=&e_4e_9f_{u_2}+ e_3e_5f_{u_3}\\

e_6(e^{v_5}-e^{v_6})=&-e_3e_4f_{u_4}+ e_3e_5f_{u_3}
\end{array}$$
hence
$$
\begin{array}{rl}
  e_5e_6e_9f =& e_6e_9(e_5e_9-e_6e_8)f_{u_1}+  \\
     & e_9(\cancel{e_4e_5e_9}-e_4e_6e_8+e_5e_6e_7-\cancel{e_4e_5e_9})f_{u_2}+ \\
     &e_5(\cancel{e_3e_5e_9}-e_3e_6e_8-\cancel{e_3e_5e_9}+e_3e_5e_9)f_{u_3}\\
     &e_3e_5(e_6e_7-e_4e_9)f_{u_4}\in I_A^2,
\end{array}
$$
since, as noticed in Example \ref{e. k33 second power}, $e_5e_9-e_6e_8, -e_4e_6e_8+e_5e_6e_7, e_3e_5e_9-e_3e_6e_8, e_6e_7-e_4e_9\in I_A.$
\end{example}
In order to prove the main result of this section we need the following technical lemma.

\begin{lemma}\label{l.pi^{(t-1)}}Let $A\in\mathcal{M}_{m\times n}(\mathbb{N})$, $t\ge2$ and $v_1, \ldots, v_{t}\in \ker(A)$ integer vectors and consider the correspondent binomials $f_{v_1}, \ldots f_{v_{t}}$. Then, for any  $\beta$ a vector of non-negative integers, and $ \Sym{t}$ the symmetric group on $[t]$, we have
    $$\pi^{(t)}(e^{\beta} {f_{v_{1}}\cdots f_{v_{t}}})=\pi^{(t)}({f_{v_{1}}\cdots f_{v_{t}}})=\sum_{\omega\in \Sym{t}} v_{\omega(1)}\otimes \cdots\otimes v_{\omega(t)}.$$
Moreover,  if  $\ker(A)=\langle u_1, \ldots, u_s \rangle $ and $v_i=\lambda_{1,i} u_1 +\cdots+ \lambda_{s,i} u_s$ for each $i=1,\ldots, t,$ then 

 $$\pi^{(t)}({f_{v_{1}}\cdots f_{v_{t}}})=\sum_{p\in [s]^{t}}\left( \prod_{i=1}^{t} \lambda_{p_i,i}\right)  u_{p_1}\otimes \cdots\otimes u_{p_{t}}.  $$
\end{lemma}
\begin{proof}
First we show that \begin{equation}\label{eq. get rid of monomial}
    \pi^{(t)}(e^{\beta} {f_{v_{1}}\cdots f_{v_{t}}})=\pi^{(t)}({f_{v_{1}}\cdots f_{v_{t}}}).
\end{equation}
We have \begin{align*}
    \pi^{(t)}(e^{\beta} {f_{v_{1}}\cdots f_{v_{t}}})=&\pi^{(t)}\left( (e^{v_1^++\beta}-e^{v_1^-+\beta})(e^{v_2^+}-e^{v_2^-})\cdots (e^{v_t^+}-e^{v_t^-})\right)=\\
    =&\sum_{\gamma\in \{- 1,+1\}^{t}}|\gamma|\bigotimes_{t}(v_1^{(\gamma_1)}+\beta +v_2^{(\gamma_2)}+\cdots +v_{t}^{(\gamma_{t})})\\
 \intertext{\small where $\gamma=(\gamma_1,\cdots, \gamma_t)\in \{+1,-1\}^{t}$,  $|\gamma|= \gamma_1\cdots \gamma_{t}$ and the exponent ${(\gamma_i)}$ is $+/-$ when $\gamma_i$ is $+1/-1$}
 =&\sum_{\gamma\in \{- 1,+1\}^{t}}|\gamma|\bigotimes_{t}(v_1^{(\gamma_1)} +v_2^{(\gamma_2)}+\cdots +v_{t}^{(\gamma_{t})})+  \text{terms involving } \beta.
\end{align*} 
However, all the terms containing $\beta$ cancel out. Indeed,  in a tensor product of $t$ elements among $\beta, v_1^{(\gamma_1)}, \ldots, v_t^{(\gamma_t)}$ at least one index between 1 and $t$ does not appear, say $v_1^{(\gamma_1)}$. Thus the same summands but with opposite sign are found in correspondence of $\gamma =(+1,\gamma_2,\ldots, \gamma_t)$ and  $\gamma' =(-1,\gamma_2,\ldots, \gamma_t)$.

Now we compute $\pi^{(t)}\left(f_{v_{1}}\cdots f_{v_{t}}\right).$

 For  $t=2$ we have 
\begin{align*}\pi^{(2)}( {f_{v_{1}}f_{v_{2}}})&=\pi^{(2)}\left(\left(e^{v_1^+}-e^{v_1^-}\right) \left(e^{v_{2}^+}-e^{v_{2}^-}\right)\right) \\&=\pi^{(2)}\left(e^{v_1^++v_2^+}-e^{v_1^++v_2^-}-e^{v_1^-+v_2^+}+e^{v_1^-+v_2^-}\right) \\&=\bigotimes_{2}\left(v_1^++v_2^+\right)-\bigotimes_{2}\left(v_1^++v_2^-\right)-\bigotimes_{2}\left(v_1^-+v_2^+\right)+\bigotimes_{2}\left(v_1^-+v_2^-\right)\\  &=v_1\otimes v_2 +v_2\otimes v_1\end{align*}
where the last equality is obtained by a straightforward calculation. 

Let $t\ge 3$
 \begin{align*}
   (\star)\ \  \pi^{(t)}({f_{v_{1}}\cdots f_{v_{t}}})=&\pi^{(t)}( (e^{v_1^+}-e^{v_1^-})(e^{v_2^+}-e^{v_2^-})\cdots (e^{v_t^+}-e^{v_t^-}))=\\
    =&\sum_{\gamma\in \{- 1,+1\}^{t}}|\gamma|\bigotimes_{t}\left(v_1^{(\gamma_1)}+v_2^{(\gamma_2)}+\cdots +v_{t}^{(\gamma_{t})}\right)\\
 \intertext{ below, to shorten the notation, we set $\gamma'=(\gamma_3,\ldots, \gamma_t)$ and $w_{\gamma'}=v_3^{(\gamma_3)}+\cdots +v_{t}^{(\gamma_{t})}$ }
 =&\sum_{\gamma'\in \{- 1,+1\}^{t-2}}|\gamma'|\left(v_1^+\otimes\bigotimes_{t-1}(v_1^{+} +v_2^{+}+w_{\gamma'})-v_1^+\otimes\bigotimes_{t-1}(v_1^{+} +v_2^{-}+w_{\gamma'})\right.\\
 &\hspace{15pt}- \left. v_1^-\otimes\bigotimes_{t-1}(v_1^{-} +v_2^{+}+w_{\gamma'})+v_1^-\otimes\bigotimes_{t-1}(v_1^{-} +v_2^{-}+w_{\gamma'})\right)+\cdots.
\end{align*}
The omitted terms have as a first entry a vector with index equal or greater than 2.  We are focusing on the above four terms since the computation for the other terms is analogous. We have 
$$v_1^+\hspace{-3pt}\otimes\bigotimes_{t-1}(v_1^{+} +v_2^{+}+w_{\gamma'})-v_1^+\hspace{-3pt}\otimes\bigotimes_{t-1}(v_1^{+} +v_2^{-}+w_{\gamma'}) -  v_1^-\hspace{-3pt}\otimes\bigotimes_{t-1}(v_1^{-} +v_2^{+}+w_{\gamma'})+v_1^-\hspace{-3pt}\otimes\bigotimes_{t-1}(v_1^{-} +v_2^{-}+w_{\gamma'})$$
by adding and subtracting a same term we get
 \begin{align*}
 =&\phantom{+\ } v_1^+\otimes\bigotimes_{t-1}(v_1^{+} +v_2^{+}+w_{\gamma'})-v_1^-\otimes\bigotimes_{t-1}(v_1^{+} +v_2^{+}+w_{\gamma'})\\
 &-v_1^+\otimes\bigotimes_{t-1}(v_1^{+} +v_2^{-}+w_{\gamma'})+v_1^+\otimes\bigotimes_{t-1}(v_1^{-} +v_2^{-}+w_{\gamma'}) \\
 &-  v_1^-\otimes\bigotimes_{t-1}(v_1^{-} +v_2^{+}+w_{\gamma'})+v_1^-\otimes\bigotimes_{t-1}(v_1^{+} +v_2^{+}+w_{\gamma'})\\
 &+v_1^-\otimes\bigotimes_{t-1}(v_1^{-} +v_2^{-}+w_{\gamma'})-v_1^+\otimes\bigotimes_{t-1}(v_1^{-} +v_2^{-}+w_{\gamma'})\\
 =&\phantom{+\ } v_1\otimes\bigotimes_{t-1}(v_1^+ +v_2^{+}+w_{\gamma'})\\
 &-v_1^+\otimes\bigotimes_{t-1}(v_1^{+} +v_2^{-}+w_{\gamma'})+v_1^+\otimes\bigotimes_{t-1}(v_1^{-} +v_2^{-}+w_{\gamma'}) \\
 &-  v_1^-\otimes\bigotimes_{t-1}(v_1^{-} +v_2^{+}+w_{\gamma'})+v_1^-\otimes\bigotimes_{t-1}(v_1^{+} +v_2^{+}+w_{\gamma'})\\
 &-v_1\otimes\bigotimes_{t-1}(v_1^{-} +v_2^{-}+w_{\gamma'})\\
  \intertext{  taking the sum over all the $\gamma'$,  we get}
 \sum_{\gamma'\in \{- 1,+1\}^{t-2}}|\gamma'|&\phantom{+\ } \big(v_1\otimes\bigotimes_{t-1}(v_1^+ +v_2^{+}+w_{\gamma'})-v_1\otimes\bigotimes_{t-1}(v_1^{-} +v_2^{-}+w_{\gamma'})\\
 &-v_1^+\otimes\bigotimes_{t-1}(v_1^{+} +v_2^{-}+w_{\gamma'})+v_1^+\otimes\bigotimes_{t-1}(v_1^{-} +v_2^{-}+w_{\gamma'}) \\
 &-  v_1^-\otimes\bigotimes_{t-1}(v_1^{-} +v_2^{+}+w_{\gamma'})+v_1^-\otimes\bigotimes_{t-1}(v_1^{+} +v_2^{+}+w_{\gamma'})\big)\\
 \intertext{ then by applying formulas $(\star)$ and \eqref{eq. get rid of monomial}, we get}
 =&\hspace{5pt} v_1\otimes \pi^{(t-1)}(f_{v_1+v_2}f_{v_3}\cdots f_{v_t}) -v_1^+\otimes \pi^{(t-1)}(f_{v_1}f_{v_3}\cdots f_{v_t})+v_1^-\otimes \pi^{(t-1)}(f_{v_1}f_{v_3}\cdots f_{v_t}) \\
 =&\hspace{5pt} v_1\otimes \pi^{(t-1)}(f_{v_1+v_2}f_{v_3}\cdots f_{v_t})-v_1\otimes \pi^{(t-1)}(f_{v_1}f_{v_3}\cdots f_{v_t}). \end{align*}
Thus, by induction and using the distributive property of the tensor property, we are done.

The last part of the statement follows from the previous part and since 
 \[\pi^{(t)}( {f_{v_{1}}\cdots f_{v_{t}}})=  
 \pi^{(t)}\left(\prod_{i=1}^{t} \left(\lambda_{1,i} u_1 +\cdots+ \lambda_{s,i} u_s\right)\right) =\sum_{p\in [s]^{t}}\left( \prod_{i=1}^{t} \lambda_{p_i,i}\right)  u_{p_1}\otimes \cdots\otimes u_{p_{t}}.\qedhere\]
\end{proof}

\begin{theorem}\label{t. sat t-th power}
    Let $I_A$ be a toric ideal of $R=\K[e_{1}, \dots, e_{n}]$ and ${\bf m}=e_{1}\cdots e_{n}$. Then $$\left(I_A^{t}:{\bf m}^{\infty}\right)=I_A^{(t)}.$$
\end{theorem}
\begin{proof} 
Since $I_A$ is a prime ideal, the fact that $\left(I_A^{t}:{\bf m}^{\infty}\right)\subseteq I_A^{(t)}$ follows from Remark 1.16 in \cite{G18} which claims 
   $$I_A^{(t)}=\{f \in R\ |\ fg \in I_A^t\ \text{for some }\ g \notin I_A \}.$$ 
Now we show that  $\left(I_A^{t}:{\bf m}^{\infty}\right)\supseteq I_A^{(t)}$.

By Lemma \ref{l. generators in same fiber (t)}, there is a minimal set of generators for $I_A^{(t)}$ which are $\varphi_A$-homogeneous and with all the coefficients in $\mathbb Z$. Let $f$ be one of them. 

We proceed by induction. The case $t=2$ is Theorem \ref{t. sat second power}, so we assume $t\ge 3$. The idea of the proof will be to first express a monomial multiple of $f$ as an element in $I_{A}^{t-1}$ that is as a sum of products of $t-1$ elements in $I_{A}$. We will then show that we can factor out a polynomial that  belongs to $I_{A}$ to get a product of $t$ elements of $I_{A}$. 

Since
$f\in I^{(t)}_A\subseteq I^{(t-1)}_A$ thus, by induction, ${\bf m}^af\in I^{t-1}$ for some integer $a\ge0$, that is
\begin{equation}\label{eq.m^af}
{\bf m}^af=c_1e^{\beta_1}\underbrace{f_{v_{1}^{(1)}}\cdots f_{v_{t-1}^{(1)}}}_{f_1}+\cdots +c_re^{\beta_r}\underbrace{f_{v_{1}^{(r)}}\cdots f_{v_{t-1}^{(r)}}}_{f_{r}}  
\end{equation}
with $v_{i}^{(j)}\in \ker(A)$. Assume $\ker(A)$ minimally spanned by $u_1, \ldots, u_s$ and  $$v_{i}^{(j)}=  \lambda_{1,i}^{(j)}u_1 +\cdots +\lambda_{s,i}^{(j)}u_s \in \ker(A).$$ Then, from Corollary \ref{c.monomial multiple} and Remark \ref{r. f lambda u}, 
 for some large enough integer $b$, we get   
 
 $${\bf m}^bf_{j}=\prod_{i=1}^{t-1}\underbrace{\left(e^{\sigma_{1,i}^{(j)}}f_{\lambda_{1,i}^{(j)}u_1}+\cdots +e^{\sigma_{s,i}^{(j)}}f_{\lambda_{s,i}^{(j)}u_s}\right)}_{\boldsymbol{m}^{b_i}f_{v_i}^{(j)}}=\prod_{i=1}^{t-1}\left(g_{1,i}^{(j)}f_{u_1}+\cdots+g_{s,i}^{(j)}f_{u_s}\right)$$  
where

\begin{equation}\label{eq. g_{p,i}^{(j)}}
g_{p,i}^{(j)}=e^{\sigma_{p,i}^{(j)}}sgn(\lambda_{p,i}^{(j)})\left(\sum\limits_{h=0}^{|\lambda_{p,i}^{(j)}|-1} {e^{hu_p}}   \right)e^{(|\lambda_{p,i}^{(j)}|-1)u_p^-}.
\end{equation}
Then, we have
$${\bf m}^{a+b}f=c_1e^{\beta_1}\left(\prod_{i=1}^{t-1}\left(g_{1,i}^{(1)}f_{u_1}+\cdots+g_{s,i}^{(1)}f_{u_s}\right)\right)+\cdots +c_re^{\beta_r}\left(\prod_{i=1}^{t-1}\left(g_{1,i}^{(r)}f_{u_1}+\cdots+g_{s,i}^{(r)}f_{u_s}\right)\right).$$
For any $P\in [s]^{t-1}$ the coefficient of $F_P=\prod\limits_{i\in P}f_{u_i}$ in $\boldsymbol{m}^{a+b}f$ is 
$$G_P= \sum_{p\in \mathcal Perm(P)} \left(c_1e^{\beta_1}\left(\prod_{i=1}^{t-1}g_{p_i,i}^{(1)}\right)+ \cdots +c_re^{\beta_r}\left(\prod_{i=1}^{t-1}g_{p_i,i}^{(r)}\right) \right)=\sum_{p\in \mathcal Perm(P)} \sum_{j=1}^r \left(c_je^{\beta_j}\left(\prod_{i=1}^{t-1}g_{p_i,i}^{(j)}\right) \right)$$
here $\mathcal Perm(P)$ is the set of all the permutations of $P$, that is the set of vectors obtained permuting the entries of $P$. 
 In order to finish the proof we have to show that $G_P\in I_A$, therefore, since $F_P\in I_A^{t-1}$, we have that $\boldsymbol{m}^{a+b}f$ can be written as a sum of elements in $I_A^t.$ 
In order to prove that $G_P\in I_A$ we compute $\pi^{(0)}(G_P)$ and show that it is zero, so from Remark \ref{r. NS condition for being in I} we are done.

From Equation \eqref{eq. g_{p,i}^{(j)}} we have
$$G_P=  \sum_{p\in \mathcal Perm(P)} \sum_{j=1}^r \left(c_je^{\beta_j}\left(\prod_{i=1}^{t-1}e^{\sigma_{p_i,i}^{(j)}}sgn(\lambda_{p_i,i}^{(j)})\left(\sum\limits_{h=0}^{|\lambda_{p_i,i}^{(j)}|-1} {e^{hu_{p_i}}}   \right)e^{(|\lambda_{p_i,i}^{(j)}|-1)u_{p_i}^-}\right) \right).$$
Thus
$$\pi^{(0)}(G_P)=\sum_{p\in \mathcal Perm(P)} \sum_{j=1}^r \left(c_j\left(\prod_{i=1}^{t-1}sgn(\lambda_{p_i,i}^{(j)})\left(|\lambda_{p_i,i}^{(j)}|   \right)\right) \right)=\sum_{p\in \mathcal Perm(P)} \sum_{j=1}^r \left(c_j\prod_{i=1}^{t-1}\lambda_{p_i,i}^{(j)}\right).$$
We use Lemma \ref{l.pi^{(t-1)}} to show that this number is 0. 
From Equation \eqref{eq.m^af} and the assumption that $f\in I_A^{(t)}$, we have   ${\bf m}^af=c_1e^{\beta_1} {f_{v_{1}^{(1)}}\cdots f_{v_{t-1}^{(1)}}} +\cdots + {c_re^{\beta_r}f_{v_{1}^{(r)}}\cdots f_{v_{t-1}^{(r)}}} \in I_A^{(t)}$. Therefore, by Theorem \ref{t.t-th symb pow}, we get $\pi^{(t-1)}({\bf m}^{a}f)=0,$ that is,  
from Lemma \ref{l.pi^{(t-1)}}:
$$\pi^{(t-1)}(c_1e^{\beta_1} {f_{v_{1}^{(1)}}\cdots f_{v_{t-1}^{(1)}}} +\cdots + {c_re^{\beta_r}f_{v_{1}^{(r)}}\cdots f_{v_{t-1}^{(r)}}})=\sum_{j=1}^r c_j  \sum_{p\in [s]^{t-1}}\left( \prod_{i=1}^{t-1} \lambda_{p_i,i}^{(j)}\right)  u_{p_1}\otimes \cdots\otimes u_{p_{t-1}}  =0.$$ 
However the tensors $u_{p_1}\otimes \cdots\otimes u_{p_{t-1}}$ are linearly independent for $p\in [s]^{t-1}$, 
and then \[\sum_{j=1}^r c_j \left( \prod_{i=1}^{t-1} \lambda_{p_i,i}^{(j)}\right) =0, \ \ \forall\ p\in [s]^{t-1}.\qedhere\]
\end{proof}

The following lemma allows us to improve the computation of the symbolic powers of a toric ideal, as we explain in Remark \ref{r.improving}.
\begin{lemma}\label{l. main Saturations} Let $I$ be a homogeneous ideal. Then\ \ 
    $(I)^t:{\bf m}^\infty=(I:{\bf m}^\infty)^t:{\bf m}^\infty.$
\end{lemma} 
\begin{proof}
Since $I\subseteq I:{\bf m}^\infty$ we have $(I)^t\subseteq (I:{\bf m}^\infty)^t$ and hence $(I)^t:{\bf m}^\infty\subseteq (I:{\bf m}^\infty)^t:{\bf m}^\infty.$

    Let $f\in (I:{\bf m}^\infty)^t:{\bf m}^\infty$ be a homogeneous form,  then  we have $f{\bf m}^a\in (I:{\bf m}^\infty)^t$ for $a$ large enough, i.e.,
$$f{\bf m}^a=f_1^{(1)}\cdots f_t^{(1)}+ \cdots +f_1^{(d)}\cdots f_t^{(d)}$$
    where $f_i^{(j)}\in I:{\bf m}^\infty$, for $i=1,\ldots, t$ and $j=1,\ldots, d.$
    Then, for $b$ large enough, 
    $f_1^{(j)}\cdots f_t^{(j)}\cdot{\bf m}^b\in (I)^t$, for $j=1,\ldots, d.$
    Therefore $f{\bf m}^{a+b}\in (I)^t$ and we are done.
\end{proof}
\begin{remark}[Optimizing the computation of symbolic powers of toric ideals]\label{r.improving}
       From  \cite[Lemma 12.2]{sturmfels1996}, see Lemma \ref{l. Sturm96} in this paper, a toric ideal $I_A$ can be obtained by saturating with ${\bf m}$ the ideal $J$ generated by the binomial forms corresponding to the vectors in a basis of $\ker(A).$ 
From Theorem \ref{t. sat t-th power}, $I_A^{(t)}=I_A^{t}:{\bf m}^{\infty}$, thus as a consequence of Lemma \ref{l. main Saturations} we get $I_A^{(t)}=J^t:{\bf m}^{\infty}$. The ideal $J$ has, in general, less generators than $I_A$ and for this reason the computation of its regular power is faster. 
\end{remark}

The next example shows how Theorem \ref{t.t-th symb pow} allows to compute symbolic powers as kernel of linear maps. \begin{example}
    Consider the graph $K_{5}=\left(\{x_1,x_2,x_3, x_4, x_5\}, \{e_{ij}=\{x_i,x_j\}, i\neq j\in \{1,\ldots,5\} \}\right).$
    
    Let  $I=I_{K_{5}}\subseteq \K[e_{ij}| 1\le i<j\le 5]$ be the toric ideal of the complete graph $K_{5}.$ 
    A set of (non-minimal) generators  for the ideal $I=I_{K_{5}}$  is $\{e_{ij}e_{ab}- e_{ib}e_{aj} \}.$ 
The ideal $I^3$ is generated in degree $6.$ To show that $I^{(3)}$ has generators of degree 5, we need to focus on monomials of degree 10 in the variables $x_i$.  
In particular, we consider ${\bf m}^2=x_{1}^{2}x_{2}^{2}x_{3}^{2}x_{4}^{2}x_{5}^{2}$. 
The fiber of ${\bf m}^2$,
is the following set of monomials of degree 5 in the variables $e_{ij}$: 
\begin{align*}
\{ &e_{12}e_{13}e_{23}e_{45}^{2},\:e_{12}^{2}e_{34}e_{35}e_{45},\:e_{12}e_{13}e_{24}e_{35}e_{45}, \;e_{12}e_{14}e_{23}e_{35}e_{45},\:e_{12}e_{13}e_{25}e_{34}e_{45},\:e_{12}e_{15}e_{2,3}e_{34}e_{45}, \\
&e_{13}^{2}e_{24}e_{25}e_{45},\:e_{13}e_{14}e_{23}e_{25}e_{45},\:e_{13}e_{15}e_{23}e_{24}e_{45},\:e_{14}e_{15}e_{23}^{2}e_{45},\:e_{12}e_{14}e_{24}e_{35}^{2},\:e_{12}e_{14}e_{25}e_{34}e_{35}, \\
&e_{12}e_{15}e_{24}e_{34}e_{35},\:e_{13}e_{14}e_{24}e_{25}e_{35},\:e_{14}^{2}e_{23}e_{25}e_{35},\:e_{13}e_{15}e_{24}^{2}e_{35},\:e_{14}e_{15}e_{23}e_{24}e_{35},\:e_{12}e_{15}e_{25}e_{34}^{2}, \\
&e_{13}e_{14}e_{25}^{2}e_{34},\:e_{13}e_{15}e_{24}e_{25}e_{34},\:e_{14}e_{15}e_{23}e_{25}e_{34},\:e_{15}^{2}e_{23}e_{24}e_{34}\}.
\end{align*}
We construct the matrix which represents the map where the monomials are an ordered basis and they are mapped to there image under the $\pi^{(2)}$ map.  

\begin{minipage}{0.39\linewidth} \fontsize{5pt}{6pt}
\renewcommand\arraystretch{0.39}
{
\hspace{1in}$\left(\!\begin{array}{cccccccccccccccccccccc}
      0&2&0&0&0&0&0&0&0&0&0&0&0&0&0&0&0&0&0&0&0&0\\
      1&0&1&0&1&0&0&0&0&0&0&0&0&0&0&0&0&0&0&0&0&0\\
      0&0&0&1&0&0&0&0&0&0&1&1&0&0&0&0&0&0&0&0&0&0\\
      0&0&0&0&0&1&0&0&0&0&0&0&1&0&0&0&0&1&0&0&0&0\\
      1&0&0&1&0&1&0&0&0&0&0&0&0&0&0&0&0&0&0&0&0&0\\
      0&0&1&0&0&0&0&0&0&0&1&0&1&0&0&0&0&0&0&0&0&0\\
      0&0&0&0&1&0&0&0&0&0&0&1&0&0&0&0&0&1&0&0&0&0\\
      0&2&0&0&1&1&0&0&0&0&0&1&1&0&0&0&0&2&0&0&0&0\\
      0&2&1&1&0&0&0&0&0&0&2&1&1&0&0&0&0&0&0&0&0&0\\
      2&2&1&1&1&1&0&0&0&0&0&0&0&0&0&0&0&0&0&0&0&0\\
      0&0&0&0&0&0&2&0&0&0&0&0&0&0&0&0&0&0&0&0&0&0\\
      0&0&0&0&0&0&0&1&0&0&0&0&0&1&0&0&0&0&1&0&0&0\\
      0&0&0&0&0&0&0&0&1&0&0&0&0&0&0&1&0&0&0&1&0&0\\
      1&0&0&0&0&0&0&1&1&0&0&0&0&0&0&0&0&0&0&0&0&0\\
      0&0&1&0&0&0&2&0&1&0&0&0&0&1&0&2&0&0&0&1&0&0\\
      0&0&0&0&1&0&2&1&0&0&0&0&0&1&0&0&0&0&2&1&0&0\\
      0&0&0&0&1&0&0&0&0&0&0&0&0&0&0&0&0&0&1&1&0&0\\
      0&0&1&0&0&0&0&0&0&0&0&0&0&1&0&1&0&0&0&0&0&0\\
      2&0&1&0&1&0&2&1&1&0&0&0&0&0&0&0&0&0&0&0&0&0\\
      0&0&0&0&0&0&0&0&0&0&0&0&0&0&2&0&0&0&0&0&0&0\\
      0&0&0&0&0&0&0&0&0&1&0&0&0&0&0&0&1&0&0&0&1&0\\
      0&0&0&1&0&0&0&1&0&2&0&0&0&0&2&0&1&0&0&0&1&0\\
      0&0&0&0&0&0&0&0&0&0&1&0&0&1&0&0&1&0&0&0&0&0\\
      0&0&0&0&0&0&0&1&0&0&0&1&0&1&2&0&0&0&2&0&1&0\\
      0&0&0&0&0&0&0&0&0&0&0&1&0&0&0&0&0&0&1&0&1&0\\
      0&0&0&1&0&0&0&0&0&0&2&1&0&1&2&0&1&0&0&0&0&0\\
      0&0&0&1&0&0&0&1&0&1&0&0&0&0&0&0&0&0&0&0&0&0\\
      0&0&0&0&0&0&0&0&0&0&0&0&0&0&0&0&0&0&0&0&0&2\\
      0&0&0&0&0&1&0&0&1&2&0&0&0&0&0&0&1&0&0&0&1&2\\
      0&0&0&0&0&0&0&0&1&0&0&0&1&0&0&2&1&0&0&1&0&2\\
      0&0&0&0&0&0&0&0&0&0&0&0&0&0&0&0&0&1&0&1&1&0\\
      0&0&0&0&0&1&0&0&0&0&0&0&1&0&0&0&0&2&0&1&1&2\\
      0&0&0&0&0&0&0&0&0&0&0&0&1&0&0&1&1&0&0&0&0&0\\
      0&0&0&0&0&1&0&0&1&1&0&0&0&0&0&0&0&0&0&0&0&0\\
      0&0&0&0&0&0&0&0&0&2&0&0&0&0&0&0&0&0&0&0&0&0\\
      0&0&0&0&0&0&0&0&1&0&0&0&0&0&0&0&1&0&0&0&0&1\\
      0&0&0&0&0&0&0&1&0&0&0&0&0&0&1&0&0&0&0&0&1&0\\
      0&0&0&0&0&1&0&0&0&0&0&0&0&0&0&0&0&0&0&0&1&1\\
      0&0&0&1&0&0&0&0&0&0&0&0&0&0&1&0&1&0&0&0&0&0\\
      2&0&0&1&0&1&0&1&1&2&0&0&0&0&0&0&0&0&0&0&0&0\\
      0&0&0&0&0&0&0&0&0&0&0&0&0&0&0&2&0&0&0&0&0&0\\
      0&0&0&0&0&0&1&0&0&0&0&0&0&1&0&0&0&0&0&1&0&0\\
      0&0&0&0&0&0&0&0&0&0&0&0&1&0&0&0&0&0&0&1&0&1\\
      0&0&1&0&0&0&0&0&0&0&2&0&1&1&0&2&1&0&0&0&0&0\\
      0&0&1&0&0&0&1&0&1&0&0&0&0&0&0&0&0&0&0&0&0&0\\
      0&0&0&0&0&0&0&0&0&0&0&0&0&0&0&0&0&0&2&0&0&0\\
      0&0&0&0&1&0&0&0&0&0&0&1&0&0&0&0&0&2&2&1&1&0\\
      0&0&0&0&0&0&0&0&0&0&0&1&0&1&1&0&0&0&0&0&0&0\\
      0&0&0&0&1&0&1&1&0&0&0&0&0&0&0&0&0&0&0&0&0&0\\
      0&0&0&0&0&0&0&0&0&0&0&0&0&0&0&0&0&2&0&0&0&0\\
      0&1&0&0&0&0&0&0&0&0&0&1&1&0&0&0&0&0&0&0&0&0\\
      0&1&0&0&1&1&0&0&0&0&0&0&0&0&0&0&0&0&0&0&0&0\\
      0&0&0&0&0&0&0&0&0&0&2&0&0&0&0&0&0&0&0&0&0&0\\
      0&1&1&1&0&0&0&0&0&0&0&0&0&0&0&0&0&0&0&0&0&0\\
      4&1&1&1&1&1&1&1&1&1&0&0&0&0&0&0&0&0&0&0&0&0
      \end{array}\!\right).$}    
\end{minipage}%

\noindent The kernel of this matrix is generated by the transposed of the vector 
$$(      0, 0, 1, -1, -1, 1, 0, 1, -1, 0, 0, 1, -1, -1, 0, 0, 1, 0, 0, 1, -1, 0)$$
which corresponds via our ordered basis to 

\[
\begin{array}{rl}
    F= & e_{12}e_{13}e_{24}e_{35}e_{45} -e_{12}e_{13}e_{25}e_{34}e_{45} -e_{12}e_{14}e_{23}e_{35}e_{45} +e_{12}e_{14}e_{25}e_{34}e_{35}  \\
     &+e_{12}e_{15}e_{23}e_{34}e_{45} -e_{12}e_{15}e_{24}e_{34}e_{35} +e_{13}e_{14}e_{23}e_{25}e_{45} -e_{13}e_{14}e_{24}e_{25}e_{35} \\
     & -e_{13}e_{15}e_{23}e_{24}e_{45}+e_{13}e_{15}e_{24}e_{25}e_{34} +e_{14}e_{15}e_{23}e_{24}e_{35} -e_{14}e_{15}e_{23}e_{25}e_{34}.
\end{array}   
  \]
Then $F\in I^{(3)}$ by Theorem \ref{t.t-th symb pow}. As noticed above $F$ has degree 5 and hence $F\notin I^3$.

This example was obtained using the following M2 code. 
{\small
\begin{verbatim}-- A function which takes an ordered basis of monomials and returns the matrix 
-- used to compute the third symbolic power 

secondFiberMatrix = A -> 
(
    M := {}; 
    for x in A do M = append(M,multiDeg x); 
    L1 := {};
    for l from 0 to (length(M)-1) do 
    (
    L2 := {}; 
    for i from 0 to length(flatten entries vars ring A_0)-2  do
    (
    L2 = append(L2,(M_l_i)*((M_l_i) - 1)); 
    for j from i+1 to length(flatten entries vars ring A_0)-1 do 
    L2 = append(L2,(M_l_i)*(M_l_j))
    ); 
    L2 = append(L2,(M_l_(length(flatten(entries(vars(ring A_0)))) - 1))*
                     (M_l_(length(flatten(entries(vars(ring A_0)))) - 1)));
    L1 = append(L1,L2)
    ); 
    transpose matrix L1
) 
\end{verbatim}

\begin{verbatim}
-- function takes a map and a multidegree in the image 
-- and returns a list of the monomials in the fiber

spaghetti = (F,m) -> 
(
    MinGens := flatten(entries(mingens(ker(F)))), L2 := {}, 
        G := flatten(entries(mingens(preimage(F,ideal(m))))); 
    for i in G do 
    if (F(i) == m) then L2 = append(L2,i); 
    M := set {}; 
    N := set {L2_0};
    while not promote(ideal(toList(M)),source F) == ideal(toList(N)) do 
    (
    Temp := N - M; 
    New = set {}; 
    for A in toList Temp do 
    (
    L := {A}; T := {}; 
    for i from 0 to length(MinGens)-1 do 
    T = join(T,{leadTerm MinGens_i, MinGens_i - leadTerm MinGens_i});
    for i from 0 to length(T)-1 do 
    (if isSubset(ideal(A),ideal(T_i)) then 
    (if odd(i) then L = append(L,substitute((A/T_i)*T_(i-1),source F)) 
    else L = append(L,substitute((A/T_i)*T_(i+1),source F))));
    New = New + set L
    ); 
    M = M + N; 
    N = N + New
    ); 
    flatten(entries(mingens(ideal(toList(M)))))
)
\end{verbatim}

\begin{verbatim}
-- function which takes a monomial and returns the multidegree 

multiDeg = m -> 
(
    R := ring m; 
    L := {}; 
    for X in flatten entries vars R do 
    (
    i := 1; while not(substitute(m/X^i,R) == 0) do 
    i = i + 1; 
    L = append(L,i-1)
    ); 
    L
)  
\end{verbatim}

\begin{verbatim}
-- A function which takes an ordered basis of monomials and returns the matrix 
-- used to compute the second symbolic power
fiberMatrix = A -> 
(   
    M := {}; 
    for x in A do 
    M = append(M,multiDeg x); 
    transpose(matrix(M))
)  
\end{verbatim}}
\end{example}

\section*{ Statements and Declarations.}

\noindent{\bf Acknowledgements:} 
The authors thank Elena Guardo for suggesting the topic and for the useful conversations.
Favacchio is member of GNSAGA-INdAM and thanks the hospitality of Università degli studi di Catania and the support of D26\_PREMIO\_GRUPPI\_RIC2023\_TRIOLO\_SALVATORE.  Keiper  was supported by the project “0-dimensional schemes, Tensor Theory and applications”—PRIN 2022-Finanziato dall’Unione europea-Next Generation EU–CUP: E53D23005670006. 
Results in this note were inspired by computations with CoCoA \cite{cocoa} and Macaulay2 \cite{macaulay2}.

\noindent{\bf Competing interests:} The authors have no  potential conflicts of interest (financial or non-financial) to declare that are relevant to the content of this article.


\begin{thebibliography}{10}

\bibitem{cocoa}
J.~Abbott, A.~M. Bigatti, and L.~Robbiano.
\newblock {CoCoA}: a system for doing {C}omputations in {C}ommutative
  {A}lgebra.
\newblock Available at \url{http://cocoa.dima.unige.it}.

\bibitem{bhaskara2023comparing}
K.~Bhaskara and A.~Van~Tuyl.
\newblock Comparing invariants of toric ideals of bipartite graphs.
\newblock {\em Proceedings of the American Mathematical Society, Series B},
  10(19):219--232, 2023.

\bibitem{bocci2016waldschmidt}
C.~Bocci, S.~Cooper, E.~Guardo, B.~Harbourne, M.~Janssen, U.~Nagel,
  A.~Seceleanu, A.~V. Tuyl, and T.~Vu.
\newblock The Waldschmidt constant for squarefree monomial ideals.
\newblock {\em Journal of Algebraic Combinatorics}, 44:875--904, 2016.

\bibitem{BH10}
C.~Bocci and B.~Harbourne.
\newblock Comparing powers and symbolic powers of ideals.
\newblock {\em Journal of Algebraic Geometry}, 19(3):399--417, 2010.

\bibitem{carlini2020ideals}
E.~Carlini, H.~T. H{\`a}, B.~Harbourne, and A.~Van~Tuyl.
\newblock {\em Ideals of powers and powers of ideals}.
\newblock Springer, 2020.

\bibitem{constantinescu2018}
A.~Constantinescu and E.~Gorla.
\newblock Gorenstein liaison for toric ideals of graphs.
\newblock {\em Journal of Algebra}, 502:249--261, 2018.

\bibitem{cooper2024splittings}
S.~M. Cooper, S.~Da~Silva, M.~Gutkin, and T.~Reimer.
\newblock Splittings for symbolic powers of edge ideals of complete graphs.
\newblock {\em Journal of Commutative Algebra}, 16(2):183--196, 2024.

\bibitem{Adam23}
M.~Cummings, S.~Da~Silva, J.~Rajchgot, and A.~Van~Tuyl.
\newblock Geometric vertex decomposition and liaison for toric ideals of
  graphs.
\newblock {\em Algebraic Combinatorics}, 6(4):965--997, 2023.

\bibitem{E95}
D.~Eisenbud.
\newblock {\em Commutative algebra: with a view toward algebraic geometry},
  volume 150 of {\em Graduate Texts in Mathematics}.
\newblock Springer-Verlag, New York, 1995.

\bibitem{F24}
G.~Favacchio.
\newblock Comparability of the total Betti numbers of toric ideals of graphs.
\newblock {\em Journal of Algebra and its Application}, 2024.
\newblock Accepted for publication.

\bibitem{FHKV}
G.~Favacchio, J.~Hofscheier, G.~Keiper, and A.~Van~Tuyl.
\newblock Splittings of toric ideals.
\newblock {\em Journal of Algebra}, 574:409--433, 2021.

\bibitem{FKV}
G.~Favacchio, G.~Keiper, and A.~Van~Tuyl.
\newblock Regularity and $h$-polynomials of toric ideals of graphs.
\newblock {\em Proceedings of the American Mathematical Society},
  148(11):4665--4677, 2020.

\bibitem{ficarra2025symbolic}
A.~Ficarra and S.~Moradi.
\newblock Symbolic powers of polymatroidal ideals.
\newblock {\em arXiv preprint arXiv:2502.19998}, 2025.

\bibitem{garcia2023robustness}
I.~Garc{\'\i}a-Marco and C.~Tatakis.
\newblock On robustness and related properties on toric ideals.
\newblock {\em Journal of Algebraic Combinatorics}, 57(1):21--52, 2023.

\bibitem{geramita2017matroid}
A.~Geramita, B.~Harbourne, J.~Migliore, and U.~Nagel.
\newblock Matroid configurations and symbolic powers of their ideals.
\newblock {\em Transactions of the American Mathematical Society},
  369(10):7049--7066, 2017.

\bibitem{macaulay2}
D.~Grayson and M.~Stillman.
\newblock Macaulay 2--a system for computation in algebraic geometry and
  commutative algebra, 1997.

\bibitem{G18}
E.~Grifo.
\newblock Symbolic powers.
\newblock {\em Lecture Notes for RTG Advanced Summer Mini-course in Commutative
  Algebra}, 2018.

\bibitem{Grifo24}
E.~Grifo.
\newblock Symbolic powers.
\newblock \url{https://eloisagrifo.github.io/SymbolicPowers.pdf}, 2024.

\bibitem{herzog2018binomial}
J.~Herzog, T.~Hibi, and H.~Ohsugi.
\newblock {\em Binomial ideals}, volume 279.
\newblock Springer, 2018.


\bibitem{HHT07}
J.~Herzog, T.~Hibi, and N.~V. Trung.
\newblock Symbolic powers of monomial ideals and vertex cover algebras.
\newblock {\em Advances in Mathematics}, 210(1):304--322, 2007.

\bibitem{K-thesis}
G.~Keiper.
\newblock {\em Toric Ideals of Finite Simple Graphs}.
\newblock PhD thesis, McMaster University, 2022.
\newblock \url{http://hdl.handle.net/11375/27878}.

\bibitem{MMN2024formula}
P.~Mantero, C.~B. Miranda-Neto, and U.~Nagel.
\newblock A formula for symbolic powers.
\newblock {\em Journal of Algebra}, 2024.

\bibitem{sturmfels1996}
B.~Sturmfels.
\newblock {\em Grobner bases and convex polytopes}, volume~8.
\newblock American Mathematical Soc., 1996.

\end{thebibliography}
\end{document}